\documentclass[12pt, reqno, twoside, letterpaper]{amsart}
\usepackage[arXivon]{short-intervals-with-a-given-number-of-primes}

\title[Short intervals with a given number of primes]
      {Short intervals with a given number of primes}

\author[T. Freiberg]{Tristan Freiberg}

\address{Department of Pure Mathematics, 
         University of Waterloo, 
         Waterloo, ON, CANADA.}

\email{tfreiberg@uwaterloo.ca}

\date{\today}

\begin{document}


\begin{abstract}
A well-known conjecture asserts that, for any given positive real 
number $\lambda$ and nonnegative integer $m$, the proportion of 
positive integers \linebreak $n \le x$ for which the interval 
$(n,n + \lambda\log n]$ contains exactly $m$ primes is 
asymptotically equal to $\lambda^m\e^{-\lambda}/m!$ as $x$ tends 
to infinity.
We show that the number of such $n$ is at least $x^{1 - o(1)}$.
\end{abstract}

\maketitle


\section{Introduction}
 \label{sec:1}

Let $\pi(x) \defeq \#\{\text{$p \le x$ : $p$ prime}\}$ denote the 
prime counting function.
One form of the prime number theorem states that, for any given 
positive real number $\lambda$, 
\[
 \frac{1}{x}
  \sum_{n \le x}
   \br{\pi(n + \lambda\log n) - \pi(n)}
    \sim 
     \lambda
      \quad 
       (x \to \infty),
\]
i.e.\ on average over $n \le x$, the interval 
$(n,n + \lambda\log n]$ contains approximately $\lambda$ primes.
As to the finer questions pertaining to the distribution of 
primes, we have little more than conjecture in the way of 
answers.
Heuristics based on Cram\'er's model suggest that, for any given 
positive real number $\lambda$ and  nonnegative integer $m$,  
\begin{equation}
 \label{eq:1.1}
   \#\{n \le x : \pi(n + \lambda \log n) - \pi(n) = m\}
    \sim
     x\lambda^m\e^{-\lambda}/m!
      \quad 
       (x \to \infty).
\end{equation}
However, before the groundbreaking work \cite{GPY} of Goldston, 
Pintz and Y{\i}ld{\i}r{\i}m (GPY), it had not even been 
established that
\begin{equation}
 \label{eq:1.2}
 \pi(n + \lambda \log n) - \pi(n) \ge m
\end{equation}
holds for infinitely many $n$ when $\lambda = 1/5$ (for instance) 
and $m = 2$.

What GPY showed is that, for arbitrarily small $\lambda$ and 
$m = 2$, \eqref{eq:1.2} holds for infinitely many $n$.
Only very recently has the breakthrough of Maynard \cite{MAY} on 
bounded gaps between primes shown that, for every choice of 
$\lambda$ and $m$, \eqref{eq:1.2} holds for infinitely many $n$.
This statement does not preclude the possibility that there are 
choices of $\lambda$ and $m$ for which
$\pi(n + \lambda \log n) - \pi(n) = m$ for at most finitely many 
$n$.
The purpose of this note is to establish the following.

\begin{theorem}
 \label{thm:1.1}
Fix any positive real number $\lambda$ and any nonnegative integer 
$m$.
If $x$ is sufficiently large in terms of $\lambda$ and $m$, then 
\begin{equation}
 \label{eq:1.3}
 \#\{n \le x : \pi(n + \lambda\log n) - \pi(n) = m\}
  \ge 
   x^{1 - \varepsilon(x)},
\end{equation}
where $\varepsilon(x)$ is a certain function that tends to zero as 
$x$ tends to infinity.
\end{theorem}


\subsection*{Notation}
 
Throughout, $\PP$ denotes the set of all primes, 
$\ind{\PP} : \NN \to \{0,1\}$ the indicator function of 
$\PP \subseteq \NN \defeq \{1,2,\ldots\}$ and $p$ a prime. 
Given $a,q \in \ZZ$, $a \pod{q}$ denotes the residue class 
$\{a + qb : b \in \ZZ\}$ (thus, $n \equiv a \pod{q}$ if and only 
if $n \in a \pod{q}$). 
Given a large real number $x$, $\log_2 x \defeq \log\log x$, 
$\log_3 x \defeq \log\log\log x$ and so on.
By $o(1)$ we mean a quantity that tends to $0$ as $x$ tends to 
infinity.
Expressions of the form $A = O(B)$, $A \ll B$ and $B \gg A$ 
denote that $|A| \le c|B|$, where $c$ is some positive constant 
(absolute unless stated otherwise); $A \asymp B$ is and 
abbreviation for $A \ll B \ll A$.
Further notation is introduced in situ.
\qed


\section{Background}
 \label{sec:2}
 
According to Cram\'er's model,%
\footnote{%
For details, we highly recommend the insightful expository 
article \cite{SOUND} of Soundararajan. 
}
the sequence $(\ind{\PP}(n))_{n \le x}$, when $x$ is large, 
behaves roughly like a sequence $(X_n)_{n \le x}$ of Bernoulli 
random variables for which $X_n = 1$ with probability $1/\log x$ 
and $X_n = 0$ with probability $1 - 1/\log x$. 
Thus, $x^{-1}\#\{n \le x : \pi(x + h) - \pi(x) = m\}$ is to be 
thought of as the probability that $X_1 + \cdots + X_h = m$. 
Letting $x$ and $h$ tend to infinity in such a way that 
$h/\log x \sim \lambda$, we get (in the limit) a Poisson 
distribution for the sum $X_1 + \cdots + X_h$. 
Hence the conjectured asymptotic \eqref{eq:1.1}, which, as was 
shown by Gallagher \cite[Theorem 1]{GAL}, would in fact follow 
from a certain uniform version of the Hardy--Littlewood prime 
tuples conjecture.

We similarly expect the normalized spacings between consecutive 
primes to follow an exponential distribution, i.e.\ if 
$d_n \defeq p_{n+1} - p_n$, where $p_n$ denotes the $n$th 
smallest prime, we have the well-known conjecture  
\[
 \frac{1}{x}
  \#\{ n \le x : d_n/\log n \in (a,b] \}
    \sim
     \int_a^b
      \e^{-t}
       \dd t
        \quad 
         (x \to \infty).
\]
(Another form of the prime number theorem states that 
$x^{-1} \sum_{n \le x} d_n/\log n \sim 1$ as $x$ tends to 
infinity, i.e.\ $d_n/\log n \approx 1$ on average over 
$n \le x$.)
However, we do not even know of any specific limit points 
of the sequence $(d_n/\log n)$, except for $0$ and $\infty$ (the 
former following from the aforementioned result of GPY, the 
latter from an old result of Westzynthius \cite{WES}).

Nevertheless, it has recently been shown%
\footnote{%
Benatar \cite[Proposition 1.2]{BEN} has since shown that the 
method of \cite{BFM} in fact yields that $25\%$ of positive real 
numbers are limit points of the sequence 
$(d_n/\log n)$.
}
\cite[Theorem 1.1]{BFM} that, in a certain sense, $12.5\%$ of 
positive real numbers are limit points of $(d_n/\log n)$. 
This note may be regarded as a continuation of \cite{BFM}, the 
results of which are utilized here in combination with the very 
general and powerful quantitative work of Maynard \cite{MAY2}.  


\section{Proof of Theorem \ref{thm:1.1}}
 \label{sec:3}

We shall consider linear functions $L$ given by $L(n) = gn + h$, 
where $g,h \in \ZZ$ (it is to be assumed that $g \ne 0$).
A finite set $\{L_1,\ldots,L_k\}$ of linear functions is  
{\em admissible} if the set of solutions modulo $p$ to 
$L_1(n)\cdots L_k(n) \equiv 0 \pod{p}$ does not form a complete 
residue system modulo $p$, for any prime $p$.
(It is to be assumed that 
$
 \prod_{1 \le i < j \le k}(g_ih_j - g_jh_i) \ne 0
$, 
i.e.\ the linear forms in $\cL$ are distinct.)
For a linear function $L$ given by $L(n) = gn + h$, and a 
positive integer $q$, let 
$\phi_L(q) \defeq \phi(|g|q)/\phi(|g|)$, where $\phi$ denotes 
the Euler function.
 
\subsection*{Maynard's theorem}

We quote a special case of \cite[Theorem 3.1]{MAY2}.

\begin{hypothesis}[$\cL,B,x,\theta$]
 \label{hyp:1}
Let $\cL$ be an admissible set of $k$ linear functions.
Let $B$ be a positive integer, let $x$ be a large real number and 
let $0 < \theta < 1$.
For each $L \in \cL$,  
\begin{equation}
 \label{eq:3.1}
 \sums[q \le x^{\theta}][(q,B) = 1]
  \max_{(L(a),q) = 1}
   \bigg|
         \sums[n \in [x,2x)][n \equiv a \pod{q}]
          \ind{\PP}(L(n))
         -
           \frac{1}{\phi_L(q)}
            \sum_{n \in [x,2x)}
             \ind{\PP}(L(n))
   \bigg|
    \ll
      \frac{\sum_{n \in [x,2x)} \ind{\PP}(L(n))}
           {(\log x)^{100k^2}}.
\end{equation}
\end{hypothesis}

\begin{theorem}[Maynard]
 \label{thm:3.1}
Let $\cL = \{L_1,\ldots,L_k\}$ be an admissible set of $k$ linear 
functions. 
Let $B$ be a positive integer, let $x$ be a large real number and 
let $0 < \theta < 1$.
Let $\alpha > 0$.
Suppose that the coefficients of 
$L_i(n) \defeq g_in + h_i \in \cL$ satisfy 
$1 \le g_i,h_i \le x^{\alpha}$ for $i = 1,\ldots,k$, 
that $k \le (\log x)^{\alpha}$ and that $1 \le B \le x^{\alpha}$.
There is a positive constant $C$, depending only on $\theta$ 
and $\alpha$, such that the following holds.
If $k \ge C$, if $\cL,B,x,\theta$ satisfy 
Hypothesis \ref{hyp:1} and if $\delta > (\log k)^{-1}$ is such 
that 
\begin{equation}
 \label{eq:3.2}
  \frac{1}{k}
   \frac{\phi(B)}{B}
    \sum_{L \in \cL}
     \frac{\phi(g_i)}{g_i}
      \sum_{n \in [x,2x)} \ind{\PP}(L(n))
       \ge 
        \delta 
         \frac{x}{\log x},
\end{equation}
then 
\begin{equation}
 \label{eq:3.3}
 \#
  \Br{
       n \in [x,2x) : 
       \#(\{L_1(n),\ldots,L_k(n)\} \cap \PP) 
           \ge C^{-1}\delta\log k
     }
   \gg
    \frac{x}{(\e^C\log x)^k}.
\end{equation}
\textup{(}%
The implicit constant in \eqref{eq:3.1} may depend only on 
$\theta$ and $\alpha$, and that in \eqref{eq:3.3} 
depends at most on $\theta$ and $\alpha$.%
\textup{)}
\end{theorem}

\subsection*{A level of distribution result}

We need to show that Hypothesis \ref{hyp:1} holds for certain 
choices of $\cL,B,x,\theta$. 
We defer proof of the following result to \S\ref{sec:4}.

\begin{lemma}
 \label{lem:3.2}
Fix a positive integer $k$ and let $\cL = \{L_1,\ldots,L_k\}$ be 
an admissible set of $k$ linear functions. 
Let $B$ be a positive integer and let $x$ be a large real number.
Let $c$ be a positive, absolute constant and let 
$\eta \defeq c/(500k^2)$.
Suppose that the coefficients of $L_i(n) \defeq g_in + h_i$ 
satisfy $g_i = g$ and $1 \le h_i \le x$ for $i = 1,\ldots,k$, 
where $g$ is a positive integer that is coprime to $B$ and 
divides $\prod_{p \le \log x^{\eta}} p$.
Then $B$ and $c$ may be chosen so that the following holds once 
$x$ is large enough in terms of $k$.
For each $L \in \cL$,
\begin{equation}
 \label{eq:3.4}
  \frac{\phi(B)}{B}
   \frac{\phi(g)}{g}\sum_{n \in [x,2x)} \ind{\PP}(L(n))
    >
     \frac{x}{2\log x}, 
\end{equation}
and
\begin{equation}
 \label{eq:3.5}
 \sums[q \le x^{1/8}][(q,B) = 1]
  \max_{(L(a),q) = 1}
   \bigg|
         \sums[n \in [x,2x)][n \equiv a \pod{q}]
          \ind{\PP}(L(n))
         -
           \frac{1}{\phi_L(q)}
            \sum_{n \in [x,2x)}
             \ind{\PP}(L(n))
   \bigg|
    \ll
      \frac{\sum_{n \in [x,2x)} \ind{\PP}(L(n))}
           {(\log x)^{100k^2}}.
\end{equation}
Moreover, $c < 1$ and either $B = 1$ or $B$ is a prime satisfying 
$\log_2 x^{\eta} \ll B \le x^{2\eta}$.
\end{lemma}

\subsection*{An Erd{\H o}s--Rankin type construction}
We quote \cite[Lemma 5.2]{BFM}.
First, some more notation and terminology: 
a finite set $\{h_1,\ldots,h_k\}$ of integers is admissible if 
the set $\{L_1,\ldots,L_k\}$ of linear forms given by 
$L_i(n) = n + h_i$, $i = 1,\ldots,k$, is admissible.
Given a real number $z \ge 1$ we let 
$[z] \defeq \{1,\ldots,\lfloor z \rfloor\}$, where 
$\lfloor z \rfloor$ denotes the greatest integer less than or 
equal to $z$.

\begin{lemma}
 \label{lem:3.3}
Fix a positive integer $k$ and $k$ nonnegative real numbers 
$\beta_1,\ldots,\beta_k$.
Suppose that $\beta_1 \le \cdots \le \beta_k$.
There is a constant $C'$, depending only on $k$ and 
$\beta_k$, such that the following holds.
Let $v,y,z$ be real numbers satisfying $v \ge 1$, $y \ge C'$ and 
\begin{equation}
 \label{eq:3.6}
  2y(1 + (1 + \beta_k)v)
   \le 2z 
    \le y(\log_2 y)/\log_3 y.
\end{equation}
Let $B$ be any positive integer such that for all prime divisors 
$l$ of $B$ \textup{(}if any\textup{)}, 
\begin{equation}
 \label{eq:3.7}
  \textstyle 
   \sum_{p \mid B, \, p \ge l} 1/p 
    \ll 1/l
     \ll 1/\log y.
\end{equation}
There exists an admissible set $\{h_1,\ldots,h_k\}$ and a sequence
$(a_p \pod{p})_{p \le y, \, p \, \nmid B}$ of residue 
classes such that 
\[
  h_i = y + \beta_i v y + O\big(y\e^{-(\log y)^{1/4}}\big)
\]
for $i = 1,\ldots,k$ and 
\[
  \textstyle
   \{h_1,\ldots,h_k\} 
  = [z]
     \, \setminus \, 
      \bigcup_{p \le y, \, p \, \nmid B} 
       a_p \pod{p}.  
\]
\end{lemma}

\subsection*{Deduction of Theorem \ref{thm:1.1}}

Fix a positive real number $\lambda$ and a nonnegative integer 
$m$.
Let $C$ be the constant of Theorem \ref{thm:3.1}, which 
depends on $\theta$ and $\alpha$. 
We will apply Theorem \ref{thm:3.1} with $\theta \defeq 1/8$ 
and $\alpha \defeq 1$, so $C$ may be regarded as absolute.
We will also apply the theorem with $\delta \defeq 1/2$.
Let $k$ be the smallest positive integer satisfying  
$k \ge C$, $k \ge \e^2$ and $k \ge \e^{2Cm}$ (i.e.\ 
$k \ge C$, $\delta \ge (\log k)^{-1}$ and 
$C^{-1}\delta \log k \ge m$).

Let $\beta_i = 2^{i - k}\lambda$, $i = 1,\ldots,k$ and let $C'$ 
be the constant of Lemma \ref{lem:3.3}, which depends on 
$k$ (hence $m$) and $\beta_k = \lambda$. 
Let $x$ be a large real number and define
\begin{equation}
 \label{eq:3.8}
 v \defeq \frac{1}{3(1 + 3\lambda)}\frac{\log_3 x}{\log_4 x}, 
 \quad 
 y \defeq 3(1 + 3\lambda)\log x \frac{\log_4 x}{\log_3 x}, 
 \quad 
 z = (1 + 3\lambda)\log x.
\end{equation}
We think of $x$ as tending to infinity, and we tacitly assume 
throughout that $x$ is ``sufficiently large'' in terms of any 
specified fixed quantity, hence ultimately in terms of $\lambda$ 
and $m$.
Thus, for instance, as is straightforward to verify, we have 
$v \ge 1$, $y \ge C'$ and \eqref{eq:3.6}.

Let $\eta$ and $B$ be as in Lemma \ref{lem:3.2}, i.e.\
$\eta \defeq c/(500k^2)$ for a certain absolute constant 
$c \in (0,1)$, and either $B = 1$ or $B$ is a prime satisfying 
$\log_2 x^{\eta} \ll B \le x^{2\eta}$.
As $\log x^{\eta} > y$, \eqref{eq:3.7} is satisfied.
Each hypothesis of Lemma \ref{lem:3.3} thus accounted for, we 
conclude that there exists an admissible set 
$\{h_1,\ldots,h_k\}$ and a sequence 
$(a_p \pod{p})_{p \le y, \, p \, \nmid B}$ of residue classes such 
that 
\[
  h_i 
   = 
    y + 2^{i - k}\lambda \log x + O\big(y\e^{-(\log y)^{1/4}}\big)
\]
for $i = 1,\ldots,k$ (we have $vy = \log x$), and 
\[
  \textstyle
   \{h_1,\ldots,h_k\} 
  = [(1 + 3\lambda)\log x]
     \, \setminus \, 
      \bigcup_{p \le y, \, p \, \nmid B} 
       a_p \pod{p}.  
\]
We work with such an admissible set and sequence of residue 
classes.
Note that
\begin{equation}
 \label{eq:3.9}
  h_k - h_1 < -1 + \lambda\log x.
\end{equation}
and
\begin{equation}
 \label{eq:3.10}
  1 < h_1 < \cdots < h_k < -1 + 2\lambda\log x
\end{equation}

We let $g \defeq \prod_{p \le y, \, p \nmid B} p$ and  
$h \pod{g}$ be the residue class modulo $g$ such that 
$h \equiv -a_p \pod{p}$ for each prime $p$ dividing $g$. 
Let us suppose that $0 \le h < g$.
We let $\cL \defeq \{L_1,\ldots,L_k\}$ be the set of linear 
functions in which $L_i(n) \defeq gn + h + h_i$ for 
$i = 1,\ldots,k$. 
It is straightforward to verify that $\cL$ is admissible, and 
that for all positive integers $n$,
\begin{equation}
 \label{eq:3.11}
 (gn + h,gn + h + (1 + 3\lambda)\log x] \cap \PP
  =
   \{L_1(n),\ldots,L_k(n)\} \cap \PP.
\end{equation}

We have $(g,B) = 1$ by definition, and as already noted, 
$\log x^{\eta} > y$, so $g$ divides 
$\prod_{p \le \log x^{\eta}} p$.
In fact, by Chebyshev's bounds and since $\eta$ is very small, we 
certainly have $0 < g, h + h_i < x$ for $i = 1,\ldots,k$.
Therefore, by Lemma \ref{lem:3.2}, $\cL,B,x$ and $\theta = 1/8$ 
satisfy Hypothesis \ref{hyp:1}, and \eqref{eq:3.2} 
holds with $\delta = 1/2$ for each $L \in \cL$.
We now invoke Theorem \ref{thm:3.1} with $\theta = 1/8$, 
$\alpha = 1$ (we have $k \le \log x$) and $\delta = 1/2$.
We've chosen $k$ so that $C^{-1}\delta\log k \ge m$, so we 
infer that 
\begin{equation}
 \label{eq:3.12}
 \#
  \Br{
       n \in [x,2x) : 
       \#(\{L_1(n),\ldots,L_k(n)\} \cap \PP) 
           \ge m
     }
   \gg
    \frac{x}{(\e^C\log x)^k}.
\end{equation}

Choose $n \in [x,2x)$ such that  
$\#(\{L_1(n),\ldots,L_k(n)\} \cap \PP) \ge m$.
Consider the intervals 
\[
  \cI_j
   \defeq 
    (N_j, N_j + \lambda\log N_j], 
     \quad 
      N_j \defeq gn + h + j, 
      \quad 
       j = 0,\ldots,\lfloor 2\lambda\log N_0 \rfloor.
\]
Now, $N_0 = x^{1 + o(1)}$, so for $j$ in the given range we have 
\[
 \cI_j
  \subseteq
   (gn + h,gn + h + (1 + 3\lambda)\log x], 
\]
and so by \eqref{eq:3.11},  
\begin{equation}
 \label{eq:3.13}
 \cI_j \cap \PP
  =
   (\cI_j \cap \{L_1(n),\ldots,L_k(n)\}) \cap \PP.  
\end{equation}
By \eqref{eq:3.9},  
\[
 L_1(n) < \cdots < L_k(n)
         < L_1(n) - 1 + \lambda\log x 
          < L_1(n) - 1 + \lambda\log N_0.
\]
Thus, if $j = h_1 - 1$ (so that $N_j = L_1(n) - 1$), then 
\begin{equation}
 \label{eq:3.14}
  \cI_j \cap \{L_1(n),\ldots,L_k(n)\} = \{L_1(n),\ldots,L_k(n)\}. 
\end{equation}
By \eqref{eq:3.10}, 
\[
 L_k(n) 
  = N_0 + h_k 
   < N_0 - 1 + 2\lambda\log x 
    < N_0 - 1 + 2\lambda\log N_0. 
\]
Thus, if $j = \lfloor2\lambda\log N_0\rfloor$, then 
\begin{equation}
 \label{eq:3.15}
 \cI_j \cap \{L_1(n),\ldots,L_k(n)\} = \emptyset.
\end{equation}
Therefore, by \eqref{eq:3.13} and \eqref{eq:3.14}, 
\[
  \#(\cI_{h_1 - 1} \cap \PP)
   =
    \#(\{L_1(n),\ldots,L_k(n)\} \cap \PP)
     \ge 
      m, 
\]
while on the other hand, by \eqref{eq:3.13} and 
\eqref{eq:3.15}, 
\[
 \#(\cI_{\lfloor2\lambda\log N_0\rfloor} \cap \PP) = 0.
\]
Now, for any $j$, if 
$\#(\cI_{j+1} \cap \PP) < \#(\cI_j \cap \PP)$ then 
$\#(\cI_{j+1} \cap \PP) = \#(\cI_j \cap \PP) - 1$.
We must conclude that there is some $j$ in the range 
$h_1 - 1 \le j \le \lfloor2\lambda\log N_0\rfloor$ 
for which $\pi(N_j + \lambda\log N_j) - \pi(N_j) = m$.

By the prime number theorem and the definition of $g$, $B$ and 
$y$, we have $g = \e^{(1 + o(1))y} > (1 + 3\lambda)\log x$.
Since $g(n + 1) + h > gn + h + (1 + 3\lambda)\log x$, no two 
values of $n$ can give rise to the same $N_j$ in this way. 
We deduce from \eqref{eq:3.12} that, with $X \defeq 4gx$, 
\[
 \#\{N \le X : \pi(N + \lambda\log N) - \pi(N) = m\}
  \gg 
   \frac{x}{(\e^C\log x)^k}.
\]
It follows that the left-hand side exceeds 
$X^{1 - \varepsilon(X)}$, where 
\[
 \varepsilon(X) \defeq (\log_4 X)^2/\log_3 X.
\]
(Recall that $\log g = (1 + o(1))y$ by the prime number theorem, 
and that, by \eqref{eq:3.8}, 
$y = 3(1 + 3\lambda)\log x\log_4 x/\log_3 x$.)
\qed


\section{Proof of Lemma \ref{lem:3.2}}
 \label{sec:4}

Lemma \ref{lem:3.2} is similar to \cite[Theorem 4.2]{BFM}.
We must nevertheless verify the details.
First, some more notation: given a positive integer $q$, 
$\chi \bmod q$, or simply $\chi$ if $q$ is clear in context, 
denotes a Dirichlet character to the modulus $q$,  
$L(s,\chi)$ denotes the $L$-function associated with it and 
$\bar{\chi}$ its complex conjugate. 
Also, $\gp{q}$ denotes the greatest prime divisor of $q$ 
($\gp{1} \defeq 1$ by convention).
We quote \cite[Lemma 4.1]{BFM}.

\begin{lemma}
 \label{lem:4.1}
Let $T \ge 3$ and let $P \ge T^{1/\log_2 T}$.
Among all primitive Dirichlet characters $\chi \bmod \ell$ to 
moduli $\ell$ satisfying $\ell \le T$ and $\gp{\ell} \le P$, there 
is at most one for which the associated $L$-function $L(s,\chi)$ 
has a zero in the region 
\begin{equation}
 \label{eq:4.1} 
  \Re(s) > 1 - c/\log P, \quad 
  |\Im(s)| \le \exp\big(\log P/\sqrt{\log T} \, \big),
\end{equation}
where $c > 0$ is a certain \textup{(}small\textup{)} absolute 
constant.
If such a character $\chi \bmod \ell$ exists, then $\chi$ is real 
and $L(s,\chi)$ has just one zero in the region \eqref{eq:4.1}, 
which is real and simple, and 
\begin{equation}
 \label{eq:4.2} 
  \gp{\ell} \gg \log \ell \gg \log_2 T.
\end{equation}
\end{lemma}

\begin{definition}
 \label{def:4.2}
For $T \ge 3$, let $\ell(T) \defeq \ell$ if the ``exceptional'' 
character $\chi \bmod \ell$, as described in 
Lemma \ref{lem:4.1}, exists; let $\ell(T) \defeq 1$ 
otherwise.
\end{definition}

\begin{proof}[Proof of Lemma \ref{lem:3.2}]
Fix a positive integer $k$.
Let $c$ be the constant of Lemma \ref{lem:4.1} and let 
$\eta \defeq c/(500k^2)$.
Let $x$ be a large real number, and let $g$ be a positive integer 
that divides $\prod_{p \le \log x^{\eta}} p$.
Note that, by Chebyshev's bounds, $g < x^{1/22}$, 
say.
(We may assume that $c$ is small.)
Let $B \defeq \ell(x^{2\eta})$, as in Definition \ref{def:4.2},
and suppose $(g,B) = 1$.
Let $h$ be an integer satisfying $1 \le h \le x$.
Suppose $(g,h) = 1$ and let $L$ denote the linear function given 
by $L(n) \defeq gn + h$.
Let $\cI_{L}(x) \defeq [gx + h,2gx + h)$.
Let $q$ denote a positive integer and $a$ an integer for which 
$(L(a),q) = 1$, noting that this implies $(L(a),gq) = 1$.
We have
\[
 \sums[n \in [x,2x)][n \equiv a \pod{q}]
  \ind{\PP}(L(n))
 =
  \sums[n \in \cI_{L}(x)][n \equiv L(a) \pod{gq}]
   \ind{\PP}(n)
 =
  \frac{1}{\phi(gq)}
   \sum_{n \in \cI_L(x)}
    \ind{\PP}(n)
    +
     \Delta_L(x;q,a), 
\]
where 
\[
 \Delta_L(x;q,a)
  \defeq 
   \sums[n \in \cI_{L}(x)][n \equiv L(a) \pod{gq}]
    \ind{\PP}(n) 
   -
  \frac{1}{\phi(gq)}
   \sum_{n \in \cI_L(x)}
    \ind{\PP}(n).    
\]
We will show that if $x$ is large enough in terms of $k$, then 
\begin{equation}
 \label{eq:4.3}
  \sums[q \le x^{1/8}][(q,B) = 1]
   \max_{(L(a),q) = 1}
    |\Delta_L(x;q,a)|
     \ll
      \frac{gx}{\phi(g)(\log x)^{2 + 100k^2}}.
\end{equation}

Let us show how this implies Lemma \ref{lem:3.2}.
First, $|\Delta_L(x;1,1)|$ is certainly majorized by the 
left-hand side of \eqref{eq:4.3}, so
\begin{equation}
 \label{eq:4.4}
  \sum_{n \in [x,2x)}
   \ind{\PP}(L(n))
  =
   \frac{1}{\phi(g)}
    \sum_{n \in \cI_L(x)}
     \ind{\PP}(n)
      + 
       \Delta_L(x;1,1)
  =
    \frac{gx\br{1 + O\br{1/\log x}}}{\phi(g)\log(gx)}
\end{equation}
by \eqref{eq:4.3} and the prime number theorem.
This gives rise to inequality \eqref{eq:3.4} (once $x$ is 
large enough in terms of $k$), for either $B = 1$ or, by 
\eqref{eq:4.2}, $B$ is a prime satisfying 
$B \gg \log_2 x^{\eta}$ (hence $\phi(B)/B \ge 1 - o(1)$), and 
we've already noted that $g \le x^{1/22}$ (hence 
$\log(gx) \le \frac{23}{22}\log x$).
Second, we verify that 
\begin{equation}
 \label{eq:4.5}
 \sums[n \in [x,2x)][n \equiv a \pod{q}]
  \ind{\PP}(L(n))
 -
   \frac{1}{\phi_L(q)}
    \sum_{n \in [x,2x)}
     \ind{\PP}(L(n))
    =
     \Delta_L(x;q,a)
    - 
     \frac{\phi(g)}{\phi(gq)}\Delta_L(x;1,1). 
\end{equation}
Third, again using \eqref{eq:4.3} to bound $|\Delta_L(x;1,1)|$, 
then using $\phi(gq) \ge \phi(g)\phi(q)$ and 
$\sum_{q \le x} 1/\phi(q) \ll \log x$, we obtain
\begin{equation}
 \label{eq:4.6}
 \sums[q \le x^{1/8}][(q,B) = 1]
  \frac{\phi(g)}{\phi(gq)}|\Delta_L(x;1,1)|
   \ll
    \frac{gx}{(\log x)^{2 + 100k^2}}
     \hspace{-4pt}
      \sum_{q \le x^{1/8}}
       \frac{1}{\phi(gq)}
        \ll
         \frac{gx}{\phi(g)(\log x)^{1 + 100k^2}}.
\end{equation}
Fourth, we combine \eqref{eq:4.3}, \eqref{eq:4.5} and 
\eqref{eq:4.6} (applying the triangle inequality to the 
right-hand side of \eqref{eq:4.5}), obtaining 
\begin{equation}
 \label{eq:4.7}
  \sums[q \le x^{1/8}][(q,B) = 1]
   \max_{(L(a),q) = 1}
    \bigg|
      \sums[n \in [x,2x)][n \equiv a \pod{q}]
       \hspace{-1pt}
        \ind{\PP}(L(n))
       -
         \frac{1}{\phi_L(q)}
          \sum_{n \in [x,2x)}
           \hspace{-1pt}
            \ind{\PP}(L(n))
    \bigg|
     \ll
      \frac{gx}{\phi(g)(\log x)^{1 + 100k^2}}.
\end{equation}
Finally, combining \eqref{eq:4.4} with \eqref{eq:4.7} 
yields \eqref{eq:3.5}.

We now establish \eqref{eq:4.3} by paraphrasing the proof of 
\cite[Theorem 4.2]{BFM}.
Suppose $1 \le q \le x^{1/8}$.
By orthogonality of Dirichlet characters we have 
\[
 \sums[n \in [x,2x)][n \equiv a \pod{q}]
  \ind{\PP}(L(n))
 =
  \sums[n \in \cI_{L}(x)][n \equiv L(a) \pod{gq}]
   \ind{\PP}(n)
 =
  \frac{1}{\phi(gq)}
   \sum_{\chi \bmod gq}
    \bar{\chi}(L(a))
     \sum_{n \in \cI_{L}(x)}
      \chi(n)\ind{\PP}(n).
\]
Letting $\chi^*$ denote the primitive character that induces 
$\chi$, we have (for characters $\chi$ to the modulus $gq$), 
\[
 \bigg|
  \sum_{n \in \cI_{L}(x)} 
  (\chi(n) - \chi(n))\ind{\PP}(n)
 \bigg|
   \le 
    \sums[n \in \cI_{L}(x)][(n,gq) > 1]
     \ind{\PP}(n)
      \le 
       \sum_{p \mid gq} 1
        \ll
         \log(gq)
          \ll
           \log x,
\]
whence 
\[
 \sums[n \in [x,2x)][n \equiv a \pod{q}]
  \ind{\PP}(L(n))
  =
  \frac{1}{\phi(gq)}
   \sum_{\chi \bmod gq}
    \bar{\chi}(L(a))
     \sum_{n \in \cI_{L}(x)}
      \chi^*(n)\ind{\PP}(n)
     + O(\log x).
\]
For the principal character $\chi_0 \bmod gq$ we have 
$\chi_0^* \equiv 1$, and so we deduce that
\begin{equation}
 \label{eq:4.8}
  \max_{(L(a),q) = 1}|\Delta_L(x;q,a)|
   \le
    \frac{1}{\phi(gq)}
     \sums[\chi \bmod gq][\chi \ne \chi_0]
      \hspace{3pt}
       \bigg|
        \sum_{n \in \cI_{L}(x)}
         \chi^*(n)\ind{\PP}(n)
       \bigg|
        + O(\log x)     
\end{equation}

It follows from the explicit formula \cite[\S19, (13)--(14)]{DAV} 
that, for nonprincipal characters $\chi \bmod gq$, 
$2 \le T \le N$ and $T \asymp \sqrt{N}$, with $\Lambda$ denoting 
the von Mangoldt function,  
\[
 \bigg|
  \sum_{n \le N} \chi(n)\Lambda(n) 
 \bigg|
  \ll 
   \sum_{|\rho| < \sqrt{N}}
    \frac{N^{\Re{(\rho)}}}{|\rho|}
     +
      \sqrt{N}(\log (gqN))^2,
\]
where the sum is over nontrivial zeros of $L(s,\chi)$ having real 
part at least $1/2$.
Since 
$
 \big|\sum_{n \le N} \chi(n)\Lambda(n)(\ind{\PP}(n) - 1)\big|
  \le 
   \sqrt{N}\log N
$,
the same bound holds if $\Lambda$ is replaced by 
$\Lambda\ind{\PP}$ and, via partial summation, $\ind{\PP}$.
Thus, since $gq \ll x^{O(1)}$ and 
$\sqrt{gx} \asymp \sqrt{gx + h}$, 
\[
  \bigg|
   \sum_{n \in \cI_{L}(x)} \chi(n)\ind{\PP}(n)
  \bigg|
   \ll
   \sum_{|\rho| < \sqrt{gx}}
    \frac{(gx)^{\Re{(\rho)}}}{|\rho|}
     +
      \sqrt{gx}(\log x)^2.    
\]
Combining this with \eqref{eq:4.8} gives   
\begin{align}
 \label{eq:4.9} 
  \begin{split}
 \max_{(L(a),q) = 1}|\Delta_{L}(x;q,a)|
 & 
  \ll
    \frac{1}{\phi(gq)}
     \sums[\chi \bmod gq][\chi \ne \chi_0]
      \hspace{4pt}
       \sumss[*][|\rho| < \sqrt{gx}]
        \frac{(gx)^{\Re(\rho)}}{|\rho|}
       + 
          \sqrt{gx}(\log x)^2
 \\
 & 
  = 
    \frac{1}{\phi(gq)}
     \sum_{d \mid gq}
      \hspace{4pt}    
       \sumss[*][\chi \bmod d]
        \hspace{3pt}
         \sum_{|\rho| < \sqrt{gx}}
          \frac{(gx)^{\Re(\rho)}}{|\rho|}
        + 
           \sqrt{gx}(\log x)^2,  
  \end{split}
\end{align}
where, in the first line, $\sumsstxt[*]$ denotes summation over 
nontrivial zeros of $L(s,\chi^*)$ having real part at least 
$1/2$, $\chi^*$ being the primitive character that induces $\chi$, 
and in the second line, $\sumsstxt[*]$ denotes summation 
over primitive characters and the innermost sum is over 
nontrivial zeros of $L(s,\chi)$ having real part at least $1/2$.

Applying \eqref{eq:4.9} and changing order of summation, 
recalling that $(g,B) = 1$, we find that 
\begin{align*}
 &
  \sums[q \le x^{1/8}][(q,B) = 1]
   \max_{(L(a),q) = 1}|\Delta_{L}(x;q,a)|
 \\
 & \hspace{30pt}
   \ll
    \sums[d \le gx^{1/8}][(d,B) = 1]
     \hspace{4pt}
      \sumss[*][\chi \bmod d]
       \hspace{3pt}
        \sum_{|\rho| < \sqrt{gx}}
         \frac{(gx)^{\Re(\rho)}}{|\rho|}
          \sums[q \le x^{1/8}][(q,B) = 1][d \mid gq]
           \frac{1}{\phi(gq)}
   +
     x^{1/8}\sqrt{gx}(\log x)^2.
\end{align*}
Writing $d = ab$ with $a = (d,g)$, for $d \mid gq$ we have 
$gq = gbc$ for some integer $c$. 
Note that $(b,g)$ divides $a$.
We have $\phi(gq) = \phi(gbc) \ge \phi(g)\phi(b)\phi(c)$, 
and as $\sum_{c \le x^{1/8}} 1/\phi(c) \ll \log x$, it 
follows that 
\begin{align}
 \label{eq:4.10}
  \begin{split}
 & 
  \sums[q \le x^{1/8}][(q,B) = 1]
   \max_{(L(a),q) = 1}|\Delta_{L}(x;q,a)|
 \\
  & \hspace{30pt}
    \ll
     \frac{\log x}{\phi(g)}
      \sum_{a \mid g}
       \sums[b \le gx^{1/8}/a][(b,g) \mid a][(b,B) = 1]
        \frac{1}{\phi(b)}
         \hspace{4pt}
          \sumss[*][\chi \bmod ab]
           \hspace{3pt}
            \sum_{|\rho| < \sqrt{gx}}
             \frac{(gx)^{\Re(\rho)}}{|\rho|}
   +
     x^{1/8}\sqrt{gx}(\log x)^2.
  \end{split}
\end{align}

If $a \mid g$ and $b \le gx^{1/8}/a$, then 
$a \in [R,2R)$ and 
$b \in [S,2S)$ for some pair $(R,S)$ of powers of $2$
satisfying $1 \le R < g$ and $1 \le RS < gx^{1/8}$.
The number of such pairs is $O((\log x)^2)$.
Note that for $b \in [S,2S)$ and $S < gx^{1/8}$ we have
$1/\phi(b) \ll (\log_2 b)/b \ll (\log x)/S$.
If $|\rho| < \sqrt{gx}$ and $1/2 \le \Re(\rho) \le 1$ then  
\[
 \Re(\rho)
  \in 
   I_m
    \defeq 
     [1/2 + m/\log(gx),1/2 + (m + 1)/\log(gx))  
\]
for some integer $m$ satisfying 
$0 \le m < \frac{1}{2}\log(gx)$, and
\[
 |\Im(\rho)|
  \in 
   J_n
    \defeq 
     [n - 1, 2n - 1) 
\]
with $n$ being some power of $2$ satisfying 
$1 \le n < \sqrt{gx}$.
The number of such pairs $(m,n)$ is $O((\log x)^2)$. 
Note that for $\rho$ with $\Re(\rho) \in I_m$ and 
$|\Im(\rho)| \in J_n$ we have 
$(gx)^{\Re(\rho)}/|\rho| \ll \sqrt{gx}\,\e^{m}/n$.
Thus, 
\begin{align}
 \label{eq:4.11}
  \begin{split}
 &  
  \frac{\log x}{\phi(g)}
   \sum_{a \mid g}
    \sums[b \le gx^{1/8}/a][(b,g) \mid a][(b,B) = 1]
     \frac{1}{\phi(b)}
      \hspace{4pt}
       \sumss[*][\chi \bmod ab]
        \hspace{3pt}
         \sum_{|\rho| < \sqrt{gx}}
          \frac{(gx)^{\Re(\rho)}}{|\rho|}
 \\
 & \hspace{30pt}
  \ll
   \frac{\sqrt{gx}(\log x)^6}{\phi(g)}
    \sup_{\substack{R < g, \, RS < gx^{1/8} \\ 2m < \log(gx) \\ n < \sqrt{gx}}}
     \frac{\e^{-m}}{nS}N^*\Big(R,S,1/2 + m/\log(gx),n - 1\Big), 
 \end{split}
\end{align}
where
\begin{equation}
 \label{eq:4.12}
 N^*(R,S,\sigma,T)
  \defeq 
   \sums[R \le a < 2R][a \mid g]
    \sums[S \le b < 2S][(b,g) \mid a][(b,B) = 1]
     \hspace{4pt}
      \sumss[*][\chi \bmod ab]
       \hspace{3pt}
        \sums[\Re(\rho) \ge \sigma][|\Im(\rho)| \le T] 1,
\end{equation}
and the innermost sum is over zeros $\rho$ of $L(s,\chi)$ in the 
given region. 

Note that, as $g$ is squarefree, every positive integer may be 
decomposed uniquely as a product $ab$ of positive integers 
$a \mid g$ and $(b,g) \mid a$, whence
\[
 N^*(R,S,\sigma,T)
  \le 
   \sum_{d < 4RS}
    \hspace{4pt}
     \sumss[*][\chi \bmod d]
      \hspace{3pt}
       \sums[\Re(\rho) \ge \sigma][|\Im(\rho)| \le T] 1.
\]
A result \cite[Theorem 12.2]{MON} of Montgomery therefore 
implies that   
\begin{equation}
 \label{eq:4.13}
  N^*(R,S,\sigma,T)
  \ll
   ((RS)^2T)^{3(1 - \sigma)/(2 - \sigma)}
    (\log(RST))^{14} 
\end{equation}
for $T \ge 2$ and $1/2 \le \sigma \le 1$.
On the right-hand side of \eqref{eq:4.11}, we partition the 
supremum set according as (i)
$0 \le m \le \frac{1}{2}\log(gx) - (44 + 200k^2)\log_2(gx)$ 
or (ii)
$
 \frac{1}{2}\log(gx) - (44 + 200k^2)\log_2(gx) 
  < m 
   < \frac{1}{2}\log(gx)
$. 

For case (i), we note that for $1/2 \le \sigma \le 1$, the 
following inequalities hold:
$1/(2 - \sigma) \le 1$, 
$6(1 - \sigma)/(2 - \sigma) \le 1 + 2(1 - \sigma)$ and 
$3(1 - \sigma)/(2 - \sigma) \le 1$.
Thus, 
$(R^2)^{3(1 - \sigma)/(2 - \sigma)} \le (R^6)^{(1 - \sigma)}$, 
$(S^2)^{3(1 - \sigma)/(2 - \sigma)} \le S(S^2)^{(1 - \sigma)}$ 
and 
$T^{3(1 - \sigma)/(2 - \sigma)} \le T$.
\linebreak
Hence \eqref{eq:4.13} implies 
$N^*(R,S,\sigma,T) \ll (R^6S^2)^{1 - \sigma}ST(\log(RST))^{14}$.
Recalling that $g < x^{1/22}$, note that if 
$R < g$ and $RS < gx^{1/8}$ then $R^6S^2 < \sqrt{gx}$.
%
It follows that  
\begin{equation}
 \label{eq:4.14}
  \frac{\e^{m}}{nS}
   N^*\Big(R,S,1/2 + m/\log(gx),n - 1\Big)
   \ll
    \e^{m/2}
     (gx)^{1/4}
     (\log x)^{14}
      \ll
       \frac{\sqrt{gx}}{(\log x)^{8 + 100k^2}}.
\end{equation}

We divide case (ii) into two sub-cases: 
(iia) $n^{3/4}S^{1/2} \ge (\log(gx))^{22 + 100k^2}$ or
(iib) $n^{3/4}S^{1/2} < (\log(gx))^{22 + 100k^2}$.
For (iia), we note that if $\sigma = 1/2 + m/\log(gx)$ 
then $\sigma \ge 10/11$ (provided $x$ is large enough in terms of 
$k$, as we assume), hence $3(1 - \sigma)/(2 - \sigma) \le 1/4$.
We have 
$
 (R^2)^{3(1 - \sigma)/(2 - \sigma)} 
  \le (R^6)^{1 - \sigma} 
   < (gx)^{(1 - \sigma)/2}
$ 
(for $R < g$) as before, and 
$(S^2T)^{3(1 - \sigma)/(2 - \sigma)} \le S^{1/2}T^{1/4}$.
By \eqref{eq:4.13} we therefore have   
\begin{equation}
 \label{eq:4.15}
  \frac{\e^{m}}{nS}
   N^*\Big(R,S,1/2 + m/\log(gx),n - 1\Big)
   \ll
    \frac{\e^{m/2}(gx)^{1/4}(\log x)^{14}}{n^{3/4}S^{1/2}}
     \ll
      \frac{\sqrt{gx}}{(\log x)^{8 + 100k^2}}.
\end{equation}

For (iib) we apply Lemma \ref{lem:4.1} to the right-hand 
side of \eqref{eq:4.12}.
Note that in this case we have $S < (\log(gx))^{44 + 200k^2}$.
Also note that, since $g \mid \prod_{p \le \log x^{\eta}} p$, 
we have $g \le x^{(1 + o(1))\eta}$ by the prime number theorem.
Recall that $\eta = c/(500k^2)$.
Thus, for $a \mid g$ and $b < 2S$ we have $ab < x^{2\eta}$ 
and $\gp{ab} < x^{2\eta/\log_2 x^{2\eta}}$ (if $x$ is large 
enough in terms of $k$, as we assume).
We find that 
\[
 1 - c\frac{\log_2 x^{2\eta}}{\log x^{2\eta}}
  <
   \frac{1}{2} + \frac{m}{\log(gx)}
    \quad 
     \text{and}
      \quad 
        n < \exp\br{\sqrt{\log x^{2\eta}}/\log_2 x^{2\eta}}
\]
when $m > \frac{1}{2}\log(gx) - (44 + 200k^2)\log_2(gx)$ and 
$n^{3/4} < (\log(gx))^{22 + 100k^2}$.
Therefore, in case (iib), $N^*(R,S,\sigma,T)$ is at most the 
number of zeros of $L(s,\chi)$ for primitive characters 
$\chi \bmod ab$, $ab < x^{2\eta} \eqdef T$, 
$\gp{ab} < T^{1/\log_2 T} \eqdef P$, 
\[
 \Re(s)
  \ge 
   1 - c/\log P,
    \quad 
     |\Im(s)|
      \le 
       \exp\br{\log P/\sqrt{\log T}}
\]
and $(ab,B) = 1$ (recall that $a \mid g$ and 
$(g,B) = (b,B) = 1$).
But by Lemma \ref{lem:4.1} and our choice of $B$, there are no 
such zeros, i.e.\
\begin{equation}
 \label{eq:4.16}
   N^*\Big(R,S,1/2 + m/\log(gx),n - 1\Big) = 0.
\end{equation}
Combining 
\eqref{eq:4.10}, 
\eqref{eq:4.11},
\eqref{eq:4.14},
\eqref{eq:4.15} 
and
\eqref{eq:4.16}  
gives \eqref{eq:4.3}.
\end{proof}


\section{Concluding remarks}
 \label{sec:5}

In \S\ref{sec:3}, we did not quote (the special case of) 
Maynard's theorem \cite[Theorem 3.1]{MAY2} in its entirety. 
It continues as follows.%
\footnote{%
We are not quoting \cite[Theorem 3.1]{MAY2} exactly here.
Rather, we are inspecting its proof \cite[(6.13) et seq.]{MAY2}.
}
If, moreover, $k \le (\log x)^{1/5}$ and 
all $L \in \cL$ have the form $gn + h_i$ with 
$|h_i| \le \upsilon\log x \ll (\log x)/(k\log k)$ and $g \ll 1$, 
then 
\begin{equation}
 \label{eq:5.1}
 \#
  \Br{
       n \in [x,2x) : 
       \#(\{L_1(n),\ldots,L_k(n)\} \cap \PP) 
           \ge C^{-1}\delta\log k
     }
   \gg
    \frac{x}{\e^{Ck^2}(\log x)^k},
\end{equation}
where on the left-hand side, 
$
 [gn,gn + \upsilon\log x] \cap \PP 
  = \{L_1(n),\ldots,L_k(n)\} \cap \PP
$.
This would lead to an improvement of Theorem \ref{thm:1.1} prima 
facie only for certain $\lambda$ and $m$ --- note here that 
$\upsilon \ll 1/(k\log k)$, so if $\upsilon \gg \lambda$ and 
$C^{-1}\delta \log k \ge m$, there is an interdependence 
between $\lambda$, $m$ and $\delta$, viz.\ 
$\lambda Cm\delta^{-1}\e^{Cm\delta^{-1}} \ll 1$.

Perhaps the right-hand side of \eqref{eq:1.3} can be improved to 
something of a quality similar to that of \eqref{eq:5.1}, 
for all $\lambda$ and $m$, via a less ad hoc proof, i.e.\ a proof 
that uses Maynard's sieve alone, and does not involve the 
Erd{\H o}s--Rankin construction.


\end{document}